\documentclass[12pt]{amsart}

\usepackage{amsmath, color} \usepackage{amssymb} \usepackage{amsthm}
\usepackage{amscd}

  \def\m{{\mathfrak{m}}}

\def\Spec{{\mathrm{Spec\; }}}

\def\Tor{\operatorname{Tor}}

\DeclareMathOperator{\cd}{cd}
\DeclareMathOperator{\pd}{pd}

\DeclareMathOperator{\st}{st}
\DeclareMathOperator{\reg}{reg}
\DeclareMathOperator{\depth}{depth}

\theoremstyle{plain}
\newtheorem{theorem}{Theorem}[section]
\newtheorem{corollary}[theorem]{Corollary}
\newtheorem{proposition}[theorem]{Proposition}

\newtheorem{definition}[theorem]{Definition}
\newtheorem{discussion}[theorem]{Discussion}
\newtheorem{lemma}[theorem]{Lemma}
\newtheorem{question}[theorem]{Question}
\newtheorem{eg}[theorem]{Example}
\newtheorem{remark}[theorem]{Remark}
\newtheorem{rem}[theorem]{Remark}

\newtheorem{algorithm}[theorem]{Algorithm}

\begin{document} 

\title{Bounds on the regularity and projective dimension of ideals associated to graphs}
\date{\today}

\author[Dao]{Hailong Dao}
\address{Department of Mathematics, University of Kansas,
Lawrence, KS 66045-7523, USA} \email{hdao@math.ku.edu}

\author[Huneke]{Craig Huneke}
\address{Department of Mathematics, University of Kansas,
Lawrence, KS 66045-7523, USA} \email{huneke@math.ku.edu}

\author[Schweig]{Jay Schweig}
\address{Department of Mathematics, University of Kansas,
Lawrence, KS 66045-7523, USA}
\email{jschweig@math.ku.edu}


\baselineskip 16pt \footskip = 32pt

\begin{abstract}
In this paper we give new upper bounds on the regularity of edge ideals 
whose resolutions are $k$-steps linear; surprisingly,  the bounds are logarithmic in the number of variables. We also
give various bounds for the projective dimension of such ideals, generalizing other recent results. By Alexander duality, our results also
apply to unmixed square-free monomial ideals of codimension two. We also discuss and connect these results
to more classical topics in commutative algebra. 
\end{abstract}

\thanks{The first and second authors gratefully acknowledge support of
the NSF during this research. The first author was supported on DMS-0834050 and the second
author was supported on DMS-1063538.}

\maketitle \markboth{H. Dao,  C.~Huneke, and J. Schweig}
{Bounds on regularity and projective dimension}

\section{Introduction}
\bigskip

In this work we give new upper bounds on the projective dimension and regularity of a certain class of square-free monomial ideals associated to graphs
which improve on existing literature. Our original motivation was purely algebraic:

\begin{question}\label{ques1}
Let $I$ be an ideal in a polynomial or regular local ring $S$ of dimension $n$, such that the quotient $S/I$ satisfies Serre's condition  $S_k$ for some integer $k \geq 2$. Can one find an upper bound on $\cd(I,S)$, the  cohomological dimension of $I$  (i.e. the least integer $q$ such that $H_I^i(S)=0$ for $i>q$)?   

\end{question}

Estimates on cohomological  dimension have been studied in depth by a number of researchers (see \cite{F}, \cite{Ha}, \cite{L1}, \cite{O}, \cite{PS}). For example,
when $I$ is an ideal in a regular local ring
$S$ of dimension $n$, Faltings \cite{F} showed that $\cd(I, S) \leq n -\left\lfloor \frac{n-1}{b}\right\rfloor$, where $b$ is the bigheight of $I$.
When  $S/I$ is normal (equivalently $R_1$ and $S_2$), it is shown in \cite[Theorem 5.1(iii)]{HL} that the cohomological dimension satisfies a stronger bound:

$$ \cd(I,S) \leq n  -\frac{2n-1}{b+1}.$$

At the time, the two authors of \cite{HL} speculated whether stronger bounds on the cohomological dimension could be found if it is required that the quotient ring $S/I$ satisfies Serre's property $S_k$ for some $k\geq 2$.  However, even now very little is known about Question  \ref{ques1}. In this paper we give upper bounds for the first non-trivial special case, when $I$ is a square-free monomial ideal of height $2$ in a polynomial ring over a field. When $I$ is squarefree monomial, its cohomological dimension and projective dimension are equal. Surprisingly, the bound we obtain is \emph{logarithmic}. To be precise, we prove (see Theorem \ref{mainThm}  and Corollary \ref{mainCor}) that in this situation, if $S/I$ satisfies Serre's condition $S_k$ for some $k\geq 2$, then 

$$ \cd(S,I) = \pd(S/I) \leq \log_{\frac{k+3}{2}}\left( \frac{(n - 1) \ln \left( \frac{k+3}{2}\right)}{k} + \frac{2}{k+3} \right) +3.$$

This suggests that the bounds in \cite{HL} may be improved, at least asymptotically and with suitable assumptions (see the discussions in Section \ref{open}).
We hope to come back to this theme in future works. 

There are also several reasons why the bounds we achieve, and the methods we use, may be of interest from a combinatorial or computational point of view.  For instance, the squarefree monomial ideals whose quotients satisfy $S_k$ have recently been studied by Murai and Terai \cite{MT}, who showed that the $h$-vectors of their corresponding Stanley-Reisner complexes display nice non-negativity behavior.

{\color{red}{}}
Considering the Alexander dual of a monomial ideal brings up other interesting connections.  Through this duality, the problem of bounding the projective dimension of monomial ideals is the same as the problem
of bounding the regularity of their Alexander duals, which are also monomial ideals.  Indeed, the condition that $S/I$ satisfies $S_k$ is equivalent to the condition that the Alexander dual of $I$ has a minimal free resolution
that is $(k-1)$-steps linear. For example, $S/I$ is $S_2$ if and only if the Alexander dual of $I$ has a linear presentation. 
The famous example of Mayr and Meyer (\cite{MM}) shows that, in general, regularity can
be doubly exponential in the degrees of generators of the ideal as well as the number of variables and equations defining $I$.  However,
a question the second author has raised for several years is whether this type of phenomenon is possible if one accounts for
the degrees of the first syzygy of the ideal (that is, if one is also given the degrees in the presentation matrix of $I$). 
\footnote{In a conversation of the second author with Jason McCullough, we realized that we know of no example  of a
homogeneous ideal $I$ in a polynomial ring in $n$ variables in which the regularity is larger than $\frac{n}{2}\cdot N$, where $N$ is
the largest regularity coming from the generators and syzygies of the ideal: $N = max\{j-i : \Tor_i(S/I,K)_j\ne 0, \quad
i\leq 2\}$.}

Thus, our main result on regularity, Theorem \ref{mainThm}, at least in the case of edge ideals of graphs, supports the general philosophy that the first syzygies have ``most" of the regularity in them. 

It should be noted that the projective dimension and regularity of Stanley-Reisner ideals, or more specifically edge ideals, have attracted a lot of attention;
see \cite{DE, FVT, HVT, HVT1, K, MS, MT} and references given therein. However, to the best of our knowledge, Question \ref{ques1} has not been investigated in this context.  Thus, the methods we develop here can be used to sharpen known results in this area. 

Yet another reason to understand monomial ideals whose quotients satisfy $S_2$ is provided by an interesting result of Kalkbrenner and Sturmfels \cite{KS} on initial ideals of homogenous prime ideals. See the discussion around Question \ref{qKS}.

Our paper is organized as follows (for any unexplained terminology, see \cite{E} or \cite{MS}).
In Section 2 we present basic material we will use throughout the paper.  Of particular interest is Lemma \ref{exact}, which gives  a recursive formula for regularity of square-free monomial ideals in general.   

Section 3 concerns regularity and
proves some results on what we call gap-free graphs. We end the section by showing the efficacy of our approach and giving a simple proof
of a theorem of Nevo (\cite{N}) that if $G$ is both claw-free and gap-free then the regularity of
the associated edge ideal $I(G)$ is at most three.

Section 4 contains the proofs of Theorems \ref{log} and \ref{mainThm}, our main results bounding the regularity of edge ideals whose resolutions are $k$-steps linear. One of the novel features of this section is Lemma \ref{bootstrap}.  Given a function of a graph's maximal degree that bounds the regularity of the associated ideal, Lemma \ref{bootstrap} gives a new bound for the regularity which is a function of the number of vertices of the graph. Thus, this lemma is of independent interest and should be useful in bounding regularity for edge ideals of other classes of graphs.

In Section 5
we turn to the problem of bounding the projective dimension of $I(G)$ for general graphs, generalizing some
recent work of Dochtermann and Engstr\"om.  Our first result gives a bound in terms of what we call the
\emph{maximal edge degree} of a graph (Theorem \ref{boundpd}), while Theorem \ref{clawfreepd} refines this result in the case when $G$ is claw-free. Our proof relies on an observation, Lemma \ref{lyu},  which gives a recursive formula for the projective dimension and should be of independent interest.

We hope that this paper will further inspire more collaborations between researchers who work in the intertwined fields of commutative algebra, combinatorial algebra and combinatorial topology. With this purpose in mind, we conclude this paper with Section 6; a discussion of open problems and further research directions.   

\thanks{ We thank David Eisenbud, Chris Francisco, Tai Huy Ha, Manoj Kummini, Gennady
Lyubeznik, Jeff Mermin, Alex Engstr\"om,  Sergey Norin, Gwyn Whieldon for valuable conversations. We also benefitted from
questions and answers on the website Mathoverflow.net.}

\bigskip

\section{Preliminaries}


Throughout this paper, we let $G$ be a finite, simple graph with vertex set $V(G)$. 
For $v, w \in V(G)$, we write $d(v,w)$ for the \emph{distance} between $v$ and $w$, the
fewest number of edges that must be traversed to travel from $v$ to $w$.  For example, $d(v, w) = 1$
if and only if $(v, w)$ is an edge of $G$.  

A subgraph $G' \subseteq G$ is called \emph{induced} if $(v,w)$ is an edge of $G'$ whenever $v$ and $w$ are vertices of $G'$ and $(v, w)$ is an edge of $G$.

The \emph{complement} of a graph $G$, for which we write $G^c$, is the graph on the same vertex set in which $(x, y)$ is an edge of $G^c$ if and only if it is not an edge of $G$.  

Finally, we let $C_k$ denote the cycle on $k$ vertices, and we let $K_{m, n}$ denote the complete bipartite graph with $m$ vertices on one side, and $n$ on the other.  

\begin{definition}
Let $G$ be a graph.  We say two edges $(w, x)$ and $(y, z)$ form a \emph{gap} in $G$ if $G$ does not have an edge with one 
endpoint in $\{w, x\}$ and the other in $\{y, z\}$.
A graph without gaps is called \emph{gap-free}.  Equivalently, $G$ is gap-free if and only if $G^c$ contains no induced $C_4$.
\end{definition}

Thus, $G$ is gap-free if and only if it does not contain two vertex-disjoint edges as an induced subgraph.  


\begin{definition}
Any graph isomorphic to $K_{1, 3}$ is called a \emph{claw}.
A graph without an induced claw is called \emph{claw-free}.
\end{definition}


If $G$ is a graph without isolated vertices, let $S$ denote the polynomial ring on the vertices of $G$ over some fixed field $k$.  Recall that the \emph{edge ideal} of $G$ is 
\[
I(G) = (xy: (x,y) \text{ is an edge of } G).
\]


\emph{Alexander duality} for monomial ideals can be defined in a number of ways (see \cite[Chapter 5]{MS}). If $I$ is a square-free  monomial ideal,
and 
\[
I = \bigcap_{\text{prime } P \supset I}P,
\]
then the
\emph{Alexander dual} of $I$, denoted $I^{\vee}$, is the square-free monomial ideal generated by all elements $x^P$, where
$x^P$ is the product of all the variables in $P$.

\begin{definition}

Let $S$ be a standard graded polynomial ring over a field $K$. 
Recall that
the (Castelnuovo-Mumford) regularity, written $\reg (M)$, of
a graded $S$-module $M$ (see, for example, \cite{E}) is given by
\[
\reg (M) = \max_m\, \{\Tor_m^S(M,K) -m\}.
\]
\end{definition}

\begin{definition}
A commutative Noetherian ring $R$ is said to satisfy  condition $S_k$ for some integer $k\geq 0$ if for all $P \in \Spec{R}$:
$$ 
\depth R_P \geq \min \{k, \dim R_P\}.
$$
\end{definition}

\begin{definition}
We say that $I(G)$ is \emph{$k$-steps linear} whenever the minimal free resolution of $I(G)$ over the polynomial ring
is linear for $k$ steps, i.e., $\Tor_{i}^S(I(G),K)_j = 0$ for all $1\leq i\leq k$ and all $j\ne i+2$. 
\end{definition}

\begin{theorem} \label{terai}(Terai, \cite{T})
Let $I \subset S$ be a square-free monomial ideal in a polynomial ring over a field. Then
$$
\reg I^{\vee} = \pd (S/I)
$$
As $I^{\vee\vee} =I$, we also have $\reg I = \pd (S/I^{\vee})$.

\end{theorem}

The following result is well-known: 

\begin{theorem}\label{inducedcycles}
Let $k \geq 2$.  Then the following are equivalent:
\begin{enumerate}
\item $S/I(G)^\vee$ satisfies Serre's condition $S_{k}$. 
\item $I(G)$ is $(k-1)$-step linear.
\item  $C_i$ is not an induced subgraph of $G^c$ for any $i$ with $4 \leq i \leq k+2$.   

\end{enumerate}
Furthermore, $\reg(I(G)) = 2$ if and only if $G^c$ contains no induced $C_i$ for any $i \geq 4$.  
\end{theorem}

\begin{proof} From \cite{Y} it follows that $S/I(G)^\vee$ is $S_{k}$ if and only if $I(G)$ has a linear resolution for at least the first $(k-1)$-steps. By \cite[Theorem 2.1]{EGHP}, this latter condition is equivalent to saying that $C_i$ is not an induced subgraph of $G^c$ for any $i$ with   $4 \leq i \leq k+2$. The last statement
follows at once from Terai's Theorem \ref{terai} (although it was first proven by Fr\"oberg in \cite{Fro}).
\end{proof}

One particularly nice feature about $k$-step linearity is that it behaves well with respect to deletion of induced subgraphs, as shown by the following.

\begin{corollary}\label{removeinduced}
Let $G$ be a graph with $I(G)$ $k$-steps linear, and let $G'$ be an induced subgraph of $G$.  Then $I(G - G')$ is $k$-steps linear. 
\end{corollary}

\begin{proof}
Because $G'$ is an induced subgraph, any induced cycle in $(G-G')^c$ would also be an induced cycle in $G^c$, and the result now follows from Theorem \ref{inducedcycles}.
\end{proof}

\begin{corollary}
The ideal $I(G)$ is $1$-step linear if and only if $G$ is gap-free.
\end{corollary}

Our proofs in the next two Sections make heavy use of the following lemma, whose last statement appears to be new and of independent interest (although in special cases such as edge ideals of chordal graphs, more information can be obtained: see, for example, \cite[3.6, 3.7]{HVT}).  

\begin{lemma}\label{exact}
Let $I \subseteq S$ be a monomial ideal, and let $x$ be a variable appearing in $I$.  Then
\[
\reg(I) \leq \max\{ \reg (I : x) + 1, \reg (I,x)\}. 
\]
Moreover, $\reg(I)$ is {\it equal} to one of
these terms.
\end{lemma}

\begin{proof}
The bound follows immediately from the well-known exact sequence
\[
0 \longrightarrow \frac{S}{(I : x)}(-1) \longrightarrow \frac{S}{I} \longrightarrow \frac{S}{(I, x)} \longrightarrow 0,
\]
where the first non-zero quotient is twisted by $-1$, since the associated map is multiplication by $x$.

To prove that equality occurs, we first show that, in general, $\reg(I,x)\leq \reg(I)$. Set 
$\reg(I,x) = m$. Then there is a square-free monomial $\tau$ such that $\beta_{j,\tau}(I,x)\ne 0$,
and $|\tau|-j = m$, where we use the multigraded Betti numbers, in homological degree $j$ and multidegree $\tau$,
(we write $|\tau|$ to mean the number of variables in the support of $\tau$). 
Observe that $(I,x) = (L,x)$, where $L$ are those monomials in $I$ which do not have $x$ in their support.
Moreover, it is easily seen that $\Tor_j^R((I,x),K)_{\tau} = \Tor_{j-1}^R(L,K)_{\tau/x}$ (see \cite[Lemma 1.3.8]{K1})
if $x$ is in the support of $\tau$, and if not, then  $\Tor_j^R((I,x),K)_{\tau} = \Tor_{j}^R(L,K)_{\tau}$. 
In the first case, $\beta_{j-1,\tau/x}(I)\ne 0$, and in the second case, $\beta_{j,\tau}(I)\ne 0$. (See
\cite[Lemma 1.3.8]{K1}, in particular (a).) In either case, the regularity of $I$ is at least $m$, proving our
claim.

Now suppose by way of contradiction that $\reg(I)$ is not equal to either term in the inequality.  From the above paragraph
it follows that $\reg (I \colon x) + 1 > \reg(I) > \reg(I,x)$. Set $m = \reg (I \colon x)$, and suppose this occurs in
homological degree $j$ and mutidegree $\sigma$, i.e., $\beta_{j,\sigma}(I \colon x)\ne 0$, and $|\sigma|-j = m$. Note that
$\sigma$ does not involve $x$ since $I:x$ has no minimal generator divisible by $x$, and also observe that
$\beta_{j,\sigma}(I\colon x) = \beta_{j+1,\sigma}(S/(I \colon x))$. 

Set $\tau = \sigma\cup \{x\}$. From the long exact sequence on Tor coming from our basic
short exact sequence, it follows that either $\beta_{j+1,\tau}(S/I)\ne 0$ or $\beta_{j+2,\tau}(S/(I,x))\ne 0$. In the
first case, the regularity of $I$ is then at least $|\tau|-j = m +1$, contradicting the assumption that $m + 1 > \reg(I)$.
In the second case, the regularity of $(I,x)$ is at least $|\tau|-(j+1) = m$. This gives a  contradiction:
$m + 1 = \reg (I \colon x) + 1 > \reg(I) > \reg(I,x) = m$.
\end{proof}

\bigskip

\section{Gap-free graphs}

In this section we observe some basic results concerning gap-free graphs and their regularity. In particular, we give a simple
proof of a result of Nevo that a claw-free and gap-free graph has regularity at most $3$.
\begin{definition}
For any graph $G$, we write $\reg(G)$ as shorthand for $\reg(I(G))$.  
\end{definition}

Recall that the \emph{star} of a vertex $x$ of $G$, for which we write $\st x$, is given by
\[
\st x = \{y \in V(G) : (x, y) \text{ is an edge of }G\} \cup \{x\}.
\]
The following lemma interprets Lemma \ref{exact} in terms of edge ideals.  Its straightforward verification is left to the reader.  

\begin{lemma}\label{removevertex}
Let $x$ be a vertex of $G$.  Then
\[
(I(G) : x) = I(G - \st x) \text{ and } (I(G), x) = I(G - x).
\]
Thus, $\reg(G) \leq \max \{ \reg(G - \st x ) + 1, \reg(G - x)\}$. Moreover, $\reg(G)$ is equal to one of
these terms.

\end{lemma}

Here and throughout we let the variables do double-duty, as they also represent vertices in the graph $G$.  

The following observation was pointed out to us by  Sergey Norin  on Mathoverflow.net (\cite{No}).

\begin{proposition}\label{distance2}
Let $G$ be gap-free, and let $x$ be a vertex of $G$ of highest degree.  Then $d(x, y) \leq 2$ for all vertices $y$ of $G$.  
\end{proposition}

\begin{proof}
Suppose otherwise.  Then there must be a vertex $y$ with $d(x, y) = 3$.  Let $x$ have degree $k$,
and list the neighbors of $x: w_1, w_2, \ldots, w_k$.  Without loss of generality, assume that $(w_1, z)$ and
$(z, y)$ are edges of $G$ (for some vertex $z$).  For any $i$ with $2 \leq i \leq k$, $(x, w_i)$ and $(z, y)$
do not form a gap in $G$.  Thus, there must be an edge with one endpoint in $\{x, w_i\}$ and one
in $\{z, y\}$.  Because $d(x, z)$ and $d(x, y)$ both exceed $1$, this edge cannot have $x$ as an endpoint. 
Similarly, $(w_i, y)$ cannot be an edge, since then we would have $d(x, y ) \leq 2$.  Thus, $(w_i, z)$
is an edge for each $i$ with $1 \leq i \leq k$ (note that we already established that $(w_1, z)$ was an edge).
Since $(z, y)$ is an edge of $G$ as well, the degree of $z$ exceeds $k$, which is a contradiction.  
\end{proof}

The above proposition together with Lemma \ref{removevertex} allow us to recover a result of Nevo with a simpler proof, see \cite[Theorem 1.2]{N}.
This theorem presumably also follows from the classification of claw-free graphs announced in \cite{CS}, but
this classification is highly non-transparent.

\begin{theorem}
Suppose $G$ is both claw-free and gap-free.  Then $\reg(G) \leq 3$.  
\end{theorem}

\begin{proof}
Let $x$ be a vertex of $G$ of highest possible degree.  By Lemma \ref{removevertex}, we know
$\reg(G) \leq \max\{ \reg(G- \st x) +1, \reg(G - x)\}$.  Note that both $G-\st x$ and $G - x$ are claw-free
and gap-free.  That $\reg(G- x) \leq 3$ is easily shown by induction on the number of vertices of $G$ (the base
case being trivial, since a simple graph with one vertex has no edges).  It remains to be shown
that $\reg(G - \st x ) \leq 2$.  By Theorem \ref{inducedcycles}, it is enough to show that $(G-\st x)^c$ contains
no induced cycle of length $\geq 4$.  Suppose on the contrary that $y_1, y_2, \ldots, y_n$ are the 
vertices of an induced cycle in $(G - \st x)^c$, where $n \geq 4$.  By Proposition \ref{distance2}, 
each $y_i$ is distance $2$ from $x$ (in $G$), so $(x, w)$ and $(w, y_1)$ are edges of $G$ for some 
vertex $w$.  Further note that $(y_2, y_n)$ is an edge of $G$, since $y_2$ and $y_n$ are non-neighbors in 
the induced cycle in $(G-\st x)^c$.  In order for the pair of edges $(x, w)$ and $(y_2, y_n)$ not to
form a gap of $G$, either $(w, y_2)$ or $(w, y_n)$ must be an edge of $G$.  Without loss, 
suppose $(w, y_2)$ is an edge.  Then the induced subgraph on $\{x, w, y_1, y_2\}$ is a claw, 
which is a contradiction.
\end{proof}

\section{Bounding the regularity}

Our main result in this section is the following, which bounds the regularity of edge ideals which are $k$-steps linear and have a given maximal degree. A general bootstrapping  process then gives upper bounds depending just on the  number of variables (see Theorem \ref{mainThm}).
There are very few examples known of edge ideals of even gap-free graphs with large regularity. The paper of Nevo and
Peeva \cite{NP} gives an example of a gap-free graph in 12 variables whose edge ideal has regularity $4$. Later \cite{N}
gave an example of a gap-free graph in 120 variables whose edge ideal has regularity $5$. It is entirely possible
that there is an absolute bound to the regularity of such edge ideals, but we have been unable to prove it. Nonetheless,
our next theorem shows that the regularity is, in the worst case, logarithmic in the maximal degree. 
\begin{theorem}\label{log}
Let $G$ be a graph such that $I(G)$ is $k$-steps linear for some $k \geq 1$, and let $d$ be the maximum degree of a vertex in $G$.  Then 
\[
\reg(I(G)) \leq \log_{\frac{k+4}{2}}\left(\frac{d}{k+1}\right) +3.
\]
\end{theorem}

Before proving Theorem \ref{log}, we introduce a process which allows us to derive bounds on graph statistics through induction on the number of vertices.  Although Corollary \ref{removeinduced} shows removal of \emph{any} induced subgraph from a graph $G$ with $I(G)$ $k$-steps linear yields a graph whose edge ideal is still $k$-steps linear, our approach works in the much more general context of \emph{stable} graph properties, which we define as follows.  

\begin{definition}
Let $P$ be a graph property.  We say that $P$ is \emph{stable} if for any graph $G$ satisfying $P$ and any maximal degree vertex $x$ of $G$, both $G - \st x$ and $G - x$ satisfy $P$.  
\end{definition}

For example, ``claw-free'' and ``gap-free'' are both stable graph properties.  Such properties allow us to bound the regularity of the
associated edge ideals in terms of the number of vertices and/or maximal vertex degree of $G$, via the following lemma.  
In some sense this lemma provides a refined and more powerful version of \cite[Theorem 3.14]{MV}.

\begin{lemma}\label{bootstrap}
Let $P$ be a stable graph property, and let $g, f: \mathbb{R}_{\geq 0} \rightarrow \mathbb{R}_{\geq 0}$ be non-decreasing functions satisfying the following.
\begin{enumerate}
\item[1)] If $G$ is a $P$-graph with no vertices of degree $>d$, $\reg(G) \leq g(d)$.
\item[2)] The function $f$ is concave-down (and thus twice-differentiable), and $f(1) \geq 2$.
\item[3)] For all $x \in \mathbb{R}_{\geq 0}$, we have $g(1/f'(x) - 1) \leq f(x)$.
\end{enumerate}
Then for any $P$-graph $G$ on $n$ vertices, we have $\reg(I(G)) \leq f(n)$.  
\end{lemma}

\begin{proof}
Let $G$ be a $P$-graph on $n$ vertices, and let $d$ be the maximal degree of a vertex in $G$.  If $d \leq 1/f'(n) - 1$, the result follows immediately since 
\[
\reg(I(G)) \leq g(d) \leq g\left( \frac{1}{f'(n)} -1\right) \leq f(n),
\]
where the middle inequality follows from the fact that $g$ is non-decreasing.  Thus, we assume $d > 1/f'(n) -1$ and induct on $n$.  For the base case $n = 1$, $G$ has no edges, so $d = 0$ and $\reg(I(G)) = 2 \leq f(1)$.  

Now let $x$ be a vertex of $G$ of maximal degree.  Because $P$ is stable, $G - \st x$ is again a $P$-graph.  Because $G - \st x$ has $n -d - 1$ vertices and $G - x$ has $n-1$ vertices, induction gives 
\[
\reg(I(G- \st x)) \leq f(n- d -1) \text{ and } \reg(I(G - x)) \leq f(n -1).
\]
By Lemma \ref{exact}, we know that $\reg(I(G)) \leq \max\{ \reg(I(G - \st x)) + 1, \reg(I(G - x))\}$, and so 
\[
\reg(I(G)) \leq \max \{ f(n - d -1) + 1, f(n-1)\}.
\]  
If $\reg(I(G)) \leq f(n-1)$ then $\reg(I(G)) \leq f(n)$, as $f$ is non-decreasing.  So, assume $\reg(I(G)) \leq f(n - d - 1) + 1$.  By the ordinary Mean Value Theorem, there exists $c < n$ with $f(n) - f(n-d-1) = (d+1)f'(c)$.  Because $d > 1/f'(n) -1$ by assumption and $f$ is concave-down, $(d+1)f'(c) \geq (d+1)f'(n) >1$, meaning $f(n) > f(n-d-1) + 1 > \reg(I(G))$. 
\end{proof}

\begin{algorithm}\label{f}
{\rm Given a function $g(x)$ as in Lemma \ref{bootstrap} (and further requiring that $g(x)$ is strictly increasing), we construct the associated function $f(x)$.  Let $h(x) = g(x -1)$, so that $h(1/f'(x)) = f(x)$, so that $f'(x)h^{-1}(f(x)) = 1$.  If we set $H(x)$ to be an antiderivative of $h^{-1}(x)$, then $f'(x)h^{-1}(f(x))$ is the derivative of $H(f(x))$, meaning $f(x) = H^{-1}$.  The additive constant of $H(x)$ can then be determined by condition $2$ of Lemma \ref{bootstrap}.}
\end{algorithm}

The following lemmas will prove helpful.  The first is a special case of Corollary \ref{removeinduced}, since a vertex and its star are both induced subgraphs.

\begin{lemma}\label{deleting}
Let $G$ be a graph such that $I(G)$ is $k$-steps linear.  Then $I(G - x)$ and $I(G - \st x)$ are both $k$-steps linear for any vertex $x$.  
\end{lemma}

\begin{lemma}\label{trim}
Let $G$ be a graph.  Then there exists a (possibly empty) sequence of vertices $w_1, w_2, \ldots, w_m$ such that the graph $G' = G - w_1 - w_2 - \ldots - w_m$ satisfies the following conditions. 
\begin{enumerate}
\item[1)] $\reg(G) \leq \reg(G')$.  
\item[2)] For any vertex $x$ of $G'$, $\reg(G') \leq \reg(G' - \st x) + 1$. 
\end{enumerate}
\end{lemma}

\begin{proof}
This is a direct consequence of Lemma \ref{removevertex}.  If $\reg(G) \leq \reg(G - \st x) + 1$ for every vertex $x$ of $G$, we may set $G' = G$.  Otherwise, there is some vertex $w_1$ with $\reg(G) > \reg(G - \st  w_1) + 1$ and $\reg(G) \leq \reg(G - w_1)$.  Now we perform the same process with $G - w_1$, and so on.  Continuing in this fashion, we must eventually reach a suitable $G'$.    
\end{proof}

\begin{definition}
We call any graph $G'$ obtained from $G$ in the above fashion a \emph{trimming} of $G$.  
\end{definition}

We are now in a position to give a proof of Theorem \ref{log}.

\begin{proof}[Proof of Theorem \ref{log}]
We define sets of graphs $G_i$ and vertices $x_i$.  First, let $G_0$ be a trimming of $G$, and let $x_0$ be a vertex of $G_0$
of maximal degree.  Now for $i > 0$, let $G_i$ be a trimming of $G_{i-1} - \st x_{i-1}$ and $x_i$ be a vertex of maximal degree in $G_i$.  Let $N_i$ be the set of neighbors of $x_i$ in $G_i$, and let $d_i = |N_i|$.  By Lemma \ref{trim} we have, for any $t \geq 0$, 
\begin{align}\label{ineq}
\reg(G) \leq \reg(G_0) \leq \reg(G_1) + 1 \leq \reg(G_2) + 2 \leq \cdots \leq \reg(G_t) +t.
\end{align}

Since $I(G)$ is $k$-steps linear, Theorem \ref{inducedcycles} gives us that $G^c$ contains no induced cycles of length $m$ for $4 \leq m \leq k+3$.  By Lemma \ref{deleting}, no $(G_i)^c$ can contain induced cycles of these lengths.  Note also that this implies that $G$ is gap-free. 

Now let $\ell$ be the greatest integer such that $(G_{\ell+1})^c$ contains an induced cycle of length $\geq k+4$.  Then $(G_{\ell  +2})^c$ contains no induced cycles of length $>3$, so Theorem \ref{inducedcycles} gives that $\reg(G_{\ell +2}) = 2$.  Let $C = \{w_1, w_2, \ldots, w_r\}$ be the vertices of some induced cycle in $(G_{\ell+1})^c$, where $r \geq k +4$ and each $(w_i , w_{i+1(\text{mod } r)})$ is an edge of $G^c$.  

For $i \leq \ell$, it must be the case that no $w_j$ is a neighbor of $x_i$ (otherwise $w_j$ would not be a vertex in $G_\ell$). 
Let $x \in N_i$.  We claim that at least $r-2$ of the vertices in $C$ are neighbors of $x$.  To see this, suppose otherwise.  Then there must be two vertices $w_j$ and $w_{j'}$ of $C$ that are not adjacent in $C$ (and so adjacent in $G$) but not adjacent to $x$.  But then $(x_i, x)$ and $(w_j, w_{j'})$ would form a gap in $G$.  

Thus, for each $i \leq \ell$, there are at least $d_i(r-2)$ edges from $N_i$ to $C$.  Because the sets $N_i$ are pairwise disjoint (by definition), there are at least $(d_0 + d_1 + \cdots + d_{\ell })(r-2)$ edges incident to vertices in $C$.  The average degree (in $G$) of a vertex in $C$ is then at least
\[
\frac{(d_0 + d_1 + \ldots + d_{\ell})(r-2)}{r}  + (r-3)\geq \frac{(d_0 + d_1 + \cdots + d_{\ell })(k+2)}{k+4} + (k+1), 
\]
where the additional additive term counts adjacencies in $C$.  Because $x_0$ is a vertex of $G_0$ of maximal degree, we have 
\[
d_0 \geq \frac{(d_0 + d_1 + \cdots + d_\ell)(k +2)}{k+4} + (k + 1) \Rightarrow d_0 \geq \frac{k+2}{2} (d_1 + \cdots + d_{\ell }) + \frac{(k+1)(k+4)}{2}
\]
Now we can apply the above argument to each of the graphs $G_i$, obtaining 
\[
d_i \geq \frac{k+2}{2}(d_{i+1} + d_{i+2} + \cdots + d_{\ell }) + \frac{(k+1)(k+4)}{2}
\]
for all $i \leq \ell $. This includes the case $i=l$, where we obtain $d_l \geq  \frac{(k+1)(k+4)}{2}$.  Letting $\alpha = \frac{k+4}{2}$, the above inequality can then be written as 
\[
d_i \geq (\alpha-1)(d_{i+1} + d_{i+2} + \cdots + d_\ell) + (k+1) \alpha.
\]
A straightforward induction argument (given that $d_\ell \geq (k + 1)\alpha$) yields $d_i \geq (k +1)\alpha^{\ell - i +1}$, and so
\[
d_0 \geq (k + 1)\alpha^{\ell+1}.  
\]
Because $d \geq d_0$ (by definition), 
\[
d \geq (k + 1) \left( \frac{k+4}{2}\right)^{\ell + 1} \Rightarrow \log_{\frac{k+4}{2}}\left( \frac{d}{k + 1}\right) - 1 \geq \ell.
\]
Since $\reg(G_{\ell +2}) = 2$ and each $G_i$ is a trimming, (\ref{ineq}) gives us
\[
\reg(G) \leq  \reg(G_{\ell + 2}) + (\ell + 2) \leq \log_{\frac{k+4}{2}}\left(\frac{d}{k+1}\right) +3.
\]
\end{proof}

We can now apply\footnote{We could also apply Algorithm \ref{f} with $g(x) = \log_{\frac{k+4}{2}}\left(\frac{x}{k + 1}\right) +3$ though the bound obtained is a good deal more complicated.} Algorithm \ref{f} with $g(x) = \log_{\frac{k+4}{2}}\left(\frac{x+1}{k + 1}\right) +3$, obtaining
\[
f(x) = \log_{\frac{k+4}{2}}\left( \frac{x \ln \left( \frac{k+4}{2}\right)}{k+1} + C \right) +3.
\] 
where $C$ is a yet-to-be-determined constant, due to the antiderivative taken in Algorithm \ref{f}.  Solving the inequality $f(1) \geq 2$ required by Lemma \ref{bootstrap}, we can set $C = \frac{2}{k+4} - \frac{\ln((k+4)/2)}{k + 1}$, giving us the following corollary.  

\begin{theorem}\label{mainThm}
Let $G$ be a graph on $n$ vertices such that $I(G)$ is $k$-steps linear.  Then 
\[
\reg(I(G)) \leq \log_{\frac{k+4}{2}}\left( \frac{(n - 1) \ln \left( \frac{k+4}{2}\right)}{k+1} + \frac{2}{k+4} \right) +3.
\]
\end{theorem}

\begin{rem}
{\rm For a graph $G$ such that $G^c$ has no induced $C_i$ for some $i\geq 4$  the number of vertices of $G$ can not exceed $(\frac{d+2}{2})^2+1$, as demonstrated by Sergey Norin on Mathoverflow (\cite{No}).  It is tempting to substitute that bound into Theorem \ref{mainThm} to get a new version of \ref{log}, from which we would derive a new version of \ref{mainThm}, then repeat. However, such a process would not yield better bounds here. }
\end{rem}

\begin{corollary}\label{mainCor}
Let $I \subset S$ be a square-free monomial ideal of height $2$ in a polynomial ring over a field. If $S/I$ satisfies condition  $S_k$ for some $k\geq 2$ then
$$ \cd(S,I) = \pd(S/I) \leq \log_{\frac{k+3}{2}}\left( \frac{(n - 1) \ln \left( \frac{k+3}{2}\right)}{k} + \frac{2}{k+3} \right) +3.$$

\end{corollary}

\bigskip

\section{Bounding the projective dimension}
\medskip

In this Section we give improved bounds on projective dimension of edge ideals. We utilize the same exact sequence from Lemma \ref{exact}.  We begin by observing the following.

\begin{lemma}\label{lyu} Let $I$ be a square-free
monomial ideal, and let $\Lambda$ be any subset of the variables. We relabel the variables so that
$\Lambda = \{x_1,...,x_i\}$. Then either 
there exists  a $j$ with $1\leq j\leq i$ such that
$\pd(S/I)  = \pd(S/((I,x_1,...,x_{j-1}) \colon x_j))$ or $\pd(S/I) = \pd(S/(I,x_1,...,x_i))$.  (Wherever applicable, we set $x_0 = 0$).
\end{lemma}

\begin{proof} Let $x = x_i$ for some $i$. We claim that
either $\pd(S/I) = \pd(S/(I,x))$ or $\pd(S/I) = \pd(S/(I:x))$. The lemma then follows easily
from this claim. To prove the claim, we suppose that $\pd(S/I)\ne \pd(S/(I  \colon x))$. 
From \cite[Lemma 1.1]{L} (see also \cite[Lemma 1.3.8]{K1}), it then follows that $\pd(S/I) > \pd(S/(I:x))$. Set $m = \pd(S/I)$. 
The long exact sequence on Tor induced by the exact sequence from Lemma \ref{exact},
together with the inequality $\pd(S/I) > \pd(S/(I:x))$, gives us $\Tor_m^S(S/(I,x),K)\ne 0$, and
$\Tor_l^S(S/(I,x),K) =  0$ for all $l > m$.  Thus, $\pd(S/(I,x)) = m$.
\end{proof}

We apply this to a graph by choosing the set of variables to be the set of neighbors of a
carefully chosen vertex. Our results improve upon \cite[Corollary 5.2]{DE}.  As before, let $G$ be a graph on $\{x_1, x_2, \ldots, x_n\}$.  

\begin{definition}
Let $(x, y)$ be an edge of $G$, and let $A = \{z : (x, z)$ and $(y, z)$ are edges of $G\}$ be the set of common neighbors of $x$ and $y$.  We define the \emph{degree} of this edge to be
\[
\deg(x) + \deg(y) - |A|.
\]
\end{definition}

Put another way, the degree of $(x, y)$ is the number of vertices adjacent to either $x$ or $y$.

\begin{theorem}\label{boundpd} Let $G$ be a graph with $n$ vertices, and let $D$ be the maximum
degree of an edge of $G$. Then $\pd(S/I(G))\leq n(1-\frac {1}{D})$.
\end{theorem}

\begin{proof} We use induction on the number of vertices of $G$. If $G$ consists of two vertices and one
edge, the bound holds. Let $(x, y)$ be an edge of maximal degree $D$, and let $x_1 = y, x_2, \ldots, x_d$ be the
neighbors of $x$. We apply Lemma \ref{lyu} to this set of vertices. First consider the
case in which $\pd(S/I(G)) = \pd(S/(I(G), x_1,...,x_d)$. We observe that $(I(G),x_1, \ldots ,x_d) =
(I(G_d), x_1, \ldots ,x_d)$, where in general we set  $G_i$ to be the induced subgraph $G-x_1-x_2- \cdots -x_i$ (minus the isolated vertices, as they have no bearing on the associated ideal).  Note also that each graph $G_i$ has maximal edge degree at
most $D$. By our choice of vertices, $G_d$ is a graph on at most $n-d-1$ vertices.
By induction, $\pd(S/I(G_d))\leq (n-d-1)(1-\frac {1}{D})$, and therefore
$\pd(S/I(G)) = \pd(S/(I(G), x_1,...,x_d) = \pd(S/(I(G_d), x_1,...,x_d) = d+ \pd(S/I(G_d)) \leq 
d + (n-d-1)(1 - \frac {1}{D})\leq n(1-\frac {1}{D})$ since $d\leq D-1$. 

We must now show that the bound holds in the second case, namely when $\pd(S/I(G)) = \pd(S/((I(G),x_1,...,x_{j-1}):x_j))$ for some $1\leq j\leq d$. In this
case,
\[
((I(G),x_1,...,x_{j-1}):x_j) = (I(G_{j-1}-\st(x_j)), x_1,...,x_{j-1}, N(x_j)),
\] 
where $N(x_j)$ is the set of neighbors of
$x_j$, and thus we have that
$$\pd(S/I(G)) = \deg_{G_{j-1}}(x_j) +  j-1 + \pd(S_{G_{j-1}-\st(x_j)}).$$ Note that $G_{j-1}-\st(x_j)$  has at most $n-j-\deg_{G_{j-1}}(x_j)$ vertices, so that by induction,
\[
\pd(S/I(G)) \leq j-1 + \deg_{G_{j-1}}(x_j)+  (n-j-\deg_{G_{j-1}}(x_j))\left(1-\frac {1}{D}\right)\leq n\left(1-\frac {1}{D}\right),
\] 
where the last inequality follows
since $j + \deg_{G_{j-1}}(x_j) \leq D$ by considering the degree of the edge $(x, x_j)$ in $G$.
\end{proof}

\begin{eg}{\rm
The bounds in Theorem \ref{boundpd} are sharp for the complete bipartite graph $G = K_{i,d}$. In this case, $n = i+d, D = i+d$, and
$\pd(S_G) = i + d-1$, which follows easily from the short exact sequence
$$0\rightarrow S_G\rightarrow S/P\oplus S/Q\rightarrow k\rightarrow 0,$$
where $P$ is the ideal generated by the $i$ variables on one side of the bipartite graph $G$, and $Q$ is generated by
the $d$-variables from the other side.
}\end{eg}

An analysis of the proof of Theorem \ref{boundpd} shows that we can often improve the result under other assumptions. 
This is due to the fact that if $\pd(S/I(G)) = \pd(S/(I(G),x_1, \ldots ,x_d)$, then in general one can often get better bounds.
On the other hand, the graphs corresponding to removing several vertices and then removing the star of another vertex may,
under suitable assumptions, have a much smaller high degree edge or vertex. We illustrate this principle in our
next theorem, which generalizes \cite[Corollary 5.3]{DE}.

\begin{theorem}\label{clawfreepd}
Let $G$ be a graph on $n$ vertices. Let $C$ be the maximum value of $( d + \left\lfloor\frac{e}{2}\right\rfloor + 1)$ where
$d$ (respectively $e$) runs through the set of all degrees of all vertices $x$ (respectively $y$) such that $(x, y)$ is an edge of $G$.
If $G$ is claw-free, then $\pd(S_G)\leq n\left(1-\frac{1}{C}\right)$. 
\end{theorem}

\begin{proof} First observe that this maximum value $C$ can never increase for induced subgraphs. We induct on $n$
to prove the statement. Fix an edge $(x, y)$ where the maximum value $C$ is obtained (that is, $\deg(x) = d, \deg(y) = e$). 
Let $\{x_1, \ldots ,x_{e}\}$ be the neighbors of $y$. We reorder these vertices as follows: Let $x_e$ be a vertex of maximal degree in the induced subgraph of $G$ with vertex set $\{x_1, \ldots, x_e\}$, and in general choose $x_i$ to be a vertex of maximal degree in the subgraph of $G$ induced by the vertices $\{x_1, x_2, \ldots, x_i\}$. We apply Lemma \ref{lyu} to the ideal $I = I(G)$ and the set of variables $\{x_1, \ldots , x_{e}\}$.   

If $\pd(S/I(G)) = \pd(S/(I(G), x_1, \ldots ,x_e))$, then by induction on $n$ (and using that 
$(I(G), x_1,\ldots ,x_e)) = (I(H), x_1, \ldots ,x_e)$, where $H = G - x_1 - x_2 - \cdots - x_e$), we have that 
\[
\pd(S/I(G))\leq (n-e-1)\left(1 - \frac{1}{C}\right) + e\leq n\left(1 - \frac{1}{C}\right).
\]
Here, note that $H$ involves only $(n-e-1)$ variables since we have isolated $y$ after deleting its
neighbors. The last inequality follows since $e\leq e + \lfloor\frac{d}{2}\rfloor + 1\leq C$ by the choice of $C$. 

Otherwise, by Lemma \ref{lyu}, we have
\[
\pd(S/I(G)) = \pd(S/((I(G),x_1, \ldots ,x_i):x_{i+1}))
\] 
for some $0\leq i\leq e-1$ (recall we set $x_0 = 0$).

Let $G'$ be the induced subgraph of $G$ with vertex set $x_1, x_2, \ldots, x_{i+1}$, and let $\delta$ be the degree of $x_{i+1}$ in $G'$.  Let $S$ be the subset of $\{x_1, x_2, \ldots, x_i\}$ of non-neighbors of $x_{i+1}$, and set $|S| = m$.  Since $G$ is claw free, the subgraph induced by $S$ must be complete (otherwise two non-neighbors of this set, together with $x_{i+1}$ and $y$, would form a claw).  Because each $x_j\in S$ has degree $\geq m-1$ in $G'$ and $x_{i+1}$ was chosen as a vertex of maximal degree, we have $\delta \geq m -1$.  Since $\delta = i - m$, this yields $\delta \geq \frac{i-1}{2}$.

Let $d_{i+1}$ be
the degree of $x_{i+1}$ in $G-\{x_1,...,x_i\}$, and observe that $d_{i+1}\leq \deg(x_{i+1}) -\delta = \deg(x_{i+1})-\frac{i-1}{2} $.
Let $H$ be the induced graph determined by deleting $x_1,\ldots ,x_i$ from $G$,
and then deleting $x_{i+1}$ and all its neighbors from $G$.

We then have
\begin{align*}
\pd(S/I(G)G) &= \pd(S/((I(G),x_1,...,x_i):x_{i+1}))\\ 
&= \pd(S/(I(H)) + i + d_{i+1} \\
&\leq (n-i-d_{i+1}-1)\left(1-\frac{1}{C}\right) + i + d_{i+1} \\
&\leq n\left(1-\frac{1}{C}\right),
\end{align*}
where the last inequality follows because $i + d_{i+1} + 1 \leq C$, which we show below.
\begin{align*}
i+d_{i+1} + 1 &\leq i + \left(\deg(x_{i+1}) -  \frac{i-1}{2}\right) + 1 \\
&= \deg(x_{i+1}) +  \frac{i-1}{2} + 2. \\
& \leq \deg(x_{i+1}) + \frac{e-2}{2} +2 \\
& = \deg(x_{i+1}) +\frac{e}{2} + 1,
\end{align*}
where we use the fact that $i\leq e-1$.  Because $i+ d_{i+1} +1 $ is an integer, it must be less than or equal to $\deg(x_{i+1}) + \left\lfloor\frac{e}{2}\right\rfloor + 1$.  Since $(x_{i+1}, y)$ is an edge of $G$, this quantity is less than or equal to $C$.  
\end{proof}

\medskip

An immediate corollary is a result due essentially to Dochtermann and Engstr\"om \cite[Corollary 5.2]{DE}, though
they have a (small) extra additive term in their bound.

\begin{corollary}\label{engstrom} Let $G$ be a graph on $n$ vertices, and assume that $d$ is the maximal degree of any vertex. Then
$\pd(S/ I(G))\leq n\left(1-\frac{1}{2d}\right)$.
\end{corollary}

\begin{proof} 
The corollary follows at once from Theorem \ref{boundpd} since the degree of any edge of $G$ is clearly at most
$2d$.
\end{proof}

\begin{remark}\label{faltings}{\rm It is worth comparing the result of Theorem \ref{boundpd} to similar bounds on the cohomological dimension
of $I(G)$. It is known that for monomial ideals $I$, the cohomological dimension $\cd(I,S)$ satisfies $\cd(I, S) =\pd(S/I)$. This
can be proved, for example, from  \cite[Theorem 0.2]{EMS}. Faltings \cite{F} gave
the most general bound for cohomological dimension, namely that
\[
\cd(I,S) \leq n -\left\lfloor \frac{n-1}{b}\right\rfloor
\]
for any ideal $I$ in a polynomial ring $S$ of dimension $n$, where $b$ is the bigheight of $I$. If $I(G)$ is the edge ideal
of a graph $G$, then it is easy to see that $d\leq b$:  If $x$ has degree $d$ with neighbors $x_1, \ldots ,x_d$, then
the fact that $x x_i\in I(G)$ for all $i$ shows that there exists a minimal prime of $I(G)$ containing $(x_1, \ldots ,x_d)$. The bound
on projective dimension we give is not completely comparable to the bound given by Faltings. 
If $d$ is close to $b$, then the bound of Faltings may be better. However, it is possible
that $d$ is much smaller than $b$ (consider a cycle, for example).  In \cite{newpaper} we further explore bounds for the projective dimension of edge ideals.}
\end{remark}

We can give fairly strong bounds relating the number of vertices in a graph and the maximum degree of an edge by combining
Theorems \ref{log} and \ref{boundpd}.

\begin{corollary}  Let $G$ be a graph such that $I(G)$ is $k$-steps linear for some $k \geq 1$, and let the regularity of
$I(G)$ be obtained at a multidegree $\sigma$, with support $H\subset G$. We consider $H$ as an induced subgraph of $G$,
and let $m$ be the number of vertices of $H$, and let $D$ be the maximal degree of an edge in $H$, and $d$ the maximal degree
of a vertex in $H$. Then
$$D \geq \frac{m}{\log_{\frac{k+4}{2}}\left(\frac{d}{k+1}\right) +3}. $$
\end{corollary}

\begin{proof}  We have that $l =\reg(I) = \reg(I_H)$, so henceforth we work just in the graph $H$. Suppose the
regularity occurs at homological degree $j$ and internal degree $m$ (by our choice of $H$), so that $m-j = l$. Then
$\reg(I(H)) + \pd(S_H) = l + \pd(S_H)\geq l + j = m$.

Since the multigraded Betti
numbers of $H$ are at most that of $G$, it follows that $I(H)$ is also  $k$-steps linear. By Theorem \ref{log},
$$\reg(H) \leq \log_{\frac{k+4}{2}}\left(\frac{d}{k+1}\right) +3.$$ On the other hand, by Theorem \ref{boundpd}, 
$\pd(S_H)\leq m(1-\frac {1}{D})$. By above,  $\pd(S_H) + \reg(H)\geq m$, and combining these inequalities yields
$$m\left(1-\frac {1}{D}\right) + \log_{\frac{k+4}{2}}\left(\frac{d}{k+1}\right) +3\geq m,$$ or equivalently
$$D \geq \frac{m}{\log_{\frac{k+4}{2}}\left(\frac{d}{k+1}\right) +3},$$ as claimed.
\end{proof}

\section{Some open questions}\label{open}
\medskip

In this Section we propose and discuss some further directions for this research. For convenience we will always assume that $I$ is a squarefree monomial
ideal inside $S = K[x_1,\cdots, x_n]$, and $G$ is a finite simple graph on $n$ vertices.

The question immediately suggested by Theorem \ref{mainThm} is whether its bound is asymptotically the best possible.  Note that we do not even know if the regularity of gap-free graphs is unbounded. The biggest regularity known currently is $5$, see \cite{NP} and \cite[Question 1.5]{W}.

\begin{question}
Is there an infinite family of gap-free graphs such that $\reg(I(G)) = O(\ln(n))$ for every $n$-vertex $G$ in this family? 
\end{question}

Given our results, it seems natural to ask if similar bounds can be obtained for arbitrary squarefree monomial ideals. 

\begin{question}\label{qreg}
Let $I \subset S$ be a square-free monomial ideal which is $k$-steps linear for some $k\geq 1$.  Is there an upper bound
for $\reg(I)$ similar to the one in Theorem \ref{mainThm}? 
\end{question}

\begin{discussion} \label{dreg}

{\rm In Question \ref{qreg}, it is likely that one has to increase $k$ to get good bounds. For example, consider the case in which $I$ is generated by cubics.
Let $G$ be any graph and $J= I(G)$. Construct an ideal $I$ as follows. Let $I$ be generated by the square-free cubics of $\m J$, where $\m = (x_1,\cdots, x_n)$.
Then $J$ is always $1$-step linear, yet the regularity of $I$ and $J$ coincide (We thank David Eisenbud for this observation). Thus, for at least the cubic case, one may need to assume at least
$2$-step linearity to get strong (logarithmic) bounds on the regularity. 

To be more specific, let $n=2l$ and let $G$ be the disjoint union of $l$ edges. Then $\reg(I(G)) = l+1$.
Let $J$ be the square-free cubics of $\m I(G)$ and $L=J^{\vee}$, the Alexander dual of $J$. Then we know $J$ is $1$-step linear, $S/L$ is $S_2$,
and $\reg(J) = \pd(S/L)
= l +1= \frac{n}{2}+1$.
Thus any upper bound better than a linear one cannot be obtained without stronger assumptions.
} 

\end{discussion}

It is worth stating the dual version of Question \ref{qreg}. In view of Discussion \ref{dreg}, we pose the following. 

\begin{question}
Let $I \subset S$ be a square-free monomial ideal of height $c$ such that $S/I$ satisfies $S_k$ for some $k\geq c$. Is there a logarithmic (in $n$) upper bound on $\pd(S/I)$? 
\end{question}

In fact, one may ask a similar question for ideals that are not necessarily monomial (see, for instance, Question \ref{ques1}).  We shall discuss here a version that is  closer to the monomial case. Suppose $S$ has characteristic $p$. Recall that a Noetherian commutative ring $R$ of characteristic $p>0$ is called $F$-{\it{pure}} if the Frobenius morphism $R \to R$ is a pure morphism.  If $I$ is a monomial ideal in $S$, it is immediately seen that $S/I$ is $F$-pure. What is less well-known is that if $S/I$ is $F$-pure, then $\cd(I,S) = \pd(S/I)$, see \cite[Theorem 4.1]{SW}.  This suggests the following question.

\begin{question}
Let $S$ be  a regular local ring or polynomial ring over a field of characteristic $p>0$ and let $I \subset S$ be an ideal such that $S/I$ is $F$-pure. Let $c$ be the codimension of $I$ and $n=\dim S$, and assume that $S/I$ satisfies  $S_k$ for some $k\geq c$.  Is there an upper bound on $\cd(I,S) = \pd(S/I)$ similar to the one in Corollary \ref{mainCor}? 

\end{question}

A weaker and more geometric version of this Question was posted by the first author on Mathoverflow.net (\cite{DMO}).

Finally, we discuss a related question motivated by an interesting result of Kalkbrenner and Sturmfels. They proved that under any monomial ordering, the reduced initial ideal of a homogeneous prime always yields a Stanley-Reisner complex which
is both pure and strongly connected (note that the $S_2$ condition is equivalent to the
corresponding Stanley-Reisner complex being locally strongly connected and pure).  So, it is natural to wonder whether one can extend the result by  Kalkbrenner and Sturmfels to give the same conclusion
locally.

\begin{question}\label{qKS}

Let $J \subset S$ be a homogenous prime ideal. Under what conditions  can one find a monomial ordering such that the reduced initial ideal $I$ of $J$ with respect to such an ordering satisfies the condition that $S/I$ is $S_2$? 

\end{question}

Our bounds on regularity show that if the initial ideal of a
homogeneous ideal $J$ in a polynomial ring $S$
is reduced and $S_2$, then the depth of $S/J$ must be extremely high, on the order of $\dim(S)-3- \text{log}(\dim S)$. In some sense,
our main result on regularity proves that initial ideals which are reduced and $S_2$ are nearly Cohen-Macaulay.
See \cite{V} for interesting results concerning the relation of cohomological dimension and other invariants between an ideal and the associated generic initial ideal.

\end{document}